\documentclass[a4paper, 11pt, english]{article}
\usepackage[utf8]{inputenc}
\usepackage[T1]{fontenc}
\usepackage{babel}
\usepackage{graphicx}
\usepackage{stmaryrd}
\usepackage[a4paper]{geometry}
\usepackage{amsmath,amsfonts,amssymb,amsthm,epsfig,epstopdf,url,array}
\usepackage{rotating}
\usepackage[colorlinks=true,citecolor=red,linkcolor=blue,pdfpagetransition=Blinds]{hyperref}
\usepackage{cleveref}
\usepackage{nameref}
\usepackage{enumitem}
\usepackage{comment}
\Crefname{paragraph}{Section}{Sections}
\usepackage{fancyhdr}
\usepackage{tikz}
\usepackage{subcaption}

\usepackage{fullpage}
\usepackage{bbm}

\newcommand{\ensemblenombre}[1]{\mathbb{#1}}
\newcommand{\N}{\ensemblenombre{N}}

\newcommand{\R}{} 
\renewcommand{\R}{\ensemblenombre{R}}
\newcommand{\C}{\ensemblenombre{C}}

\renewcommand{\leq}{\leqslant}
\renewcommand{\geq}{\geqslant}
\renewcommand{\le}{\leqslant}
\renewcommand{\ge}{\geqslant}

\newcommand{\dive}[1]{\mathrm{div}}
\renewcommand{\epsilon}{\varepsilon}

\theoremstyle{plain} 
\newtheorem{proposition}{Proposition}[section] 
\newtheorem{theorem}[proposition]{Theorem}
\newtheorem{lemma}[proposition]{Lemma}
\newtheorem{cor}[proposition]{Corollary}
\theoremstyle{definition}
\newtheorem{definition}[proposition]{Definition}
\newtheorem{remark}[proposition]{Remark}

\numberwithin{equation}{section}

\makeatletter
\let\original@addcontentsline\addcontentsline
\newcommand{\dummy@addcontentsline}[3]{}
\newcommand{\DeactivateToc}{\let\addcontentsline\dummy@addcontentsline}
\newcommand{\ActivateToc}{\let\addcontentsline\original@addcontentsline}
\makeatother

\newcommand{\be}{\begin{equation} \label}
	\newcommand{\ee}{\end{equation}}

\newcommand{\ba}{\begin{align} \label}
	\newcommand{\ea}{\end{align}}

\newcommand{\bt}{\begin{theorem} \label}
	\newcommand{\et}{\end{theorem}}

\newcommand{\bl}{\begin{lm} \label}
	\newcommand{\el}{\end{lm}}

\newcommand{\bc}{\begin{cor} \label}
	\newcommand{\ec}{\end{cor}}

\newcommand{\bp}{\begin{pr} \label}
	\newcommand{\ep}{\end{pr}}

\usepackage{graphicx}

\begin{document}

	\title{Application of uncertainty principles for decaying densities to the  observability of the Schrödinger equation}
	\author{Kévin Le Balc'h, Jiaqi Yu}
	
	\maketitle
	
\begin{abstract}
In this article, we study the Schrödinger equation posed in the Euclidean space. We prove observability inequalities for measurable sets that are thick with respect to decaying densities. The proof relies on quantitative uncertainty principles adapted to decaying densities, notably those established by Shubin, Vakilian, Wolff, and Kovrijkine.
\end{abstract}

	
	\section{Introduction}
	
	\subsection{Observability estimates of the Schrödinger equation}
	
	For $d\geq 1$, let us consider the free Schrödinger equation
	\begin{equation}
		\label{eq:SchrodingerObs}
		\left\{
		\begin{array}{ll}
			i  \partial_t u  +  \Delta u = 0 & \text{ in }  (0,+\infty) \times \R^d, \\
			u(0, \cdot) = u_0 & \text{ in } \R^d,
		\end{array}
		\right.
	\end{equation}
	where $u_0 \in L^2(\R^d)$. It is a well-known fact that $i \Delta$ generates a continuous unitary semigroup in $L^2(\R^d)$ denoted by $(e^{i t \Delta})_{t \in \R}$. 
	
	\medskip

	We start with the following definition.
	\begin{definition}[Observability inequalities]
		${}$
		\begin{itemize}
			\item For $T>0$ and $\mathcal{O} \subset \R^d$ a measurable set, we say that \eqref{eq:SchrodingerObs} is observable in $(0,T) \times \mathcal{O}$ if there exists a positive constant $C>0$ such that
			\be{eq:obs}
			\|u_0\|_{L^2(\R^d)}^2 \leq C \int_0^T \int_{\mathcal{O}} |e^{it \Delta } u_0|^2 dx dt\qquad \forall u_0 \in L^2(\R^d).
			\ee
			\item For $T>S>0$ and $\mathcal{O}_1, \mathcal{O}_2 \subset \R^d$ two measurable sets, we say that \eqref{eq:SchrodingerObs} is observable in $(\{T\} \times \mathcal{O}_1) \cup (\{S\} \times \mathcal{O}_2)$ if there exists a positive constant $C>0$ such that
			\be{eq:obstwopoints}
			\|u_0\|_{L^2(\R^d)}^2 \leq C \left(\int_{\mathcal{O}_1} |e^{iT \Delta} u_0|^2 dx + \int_{\mathcal{O}_2} |e^{iS \Delta} u_0|^2 dx \right)\qquad \forall u_0 \in L^2(\R^d).
			\ee
		\end{itemize}
	\end{definition}
	
	Note that in \eqref{eq:obs} the observation appears on the product set $(0,T)\times \mathcal{O}$, this is the usual definition of observability.  The observability at two-time points \eqref{eq:obstwopoints} has been introduced in \cite{WangWangZhang2019}, the observation here appears on two differents sets $\{T\} \times \mathcal{O}_1$ and $\{S\} \times \mathcal{O}_2$.  The two notions model the possibility of estimating the (whole) $L^2(\R^d)$-norm of the initial wave function $u(0,\cdot)=u_0$ by some restricted norm of the wave function $u=u(t,z)$. For instance if the initial datum is normalized, $\|u_0\|_{L^2(\R^d)}=1$, the expression $\int_{\mathcal{O}} |e^{it \Delta} u_0|^2 dx$ is the probability of finding the particle in the set $\mathcal{O}$ at time $t$. Having $\int_0^T \int_{\mathcal{O}} |e^{it \Delta} u_0|^2 dx dt \geq C^{-1} > 0$ for all solutions of \eqref{eq:SchrodingerObs} means that every quantum particle spends a positive fraction of time of the interval $(0,T)$ in the set $\mathcal{O}$. On the other hand, having $\int_{\mathcal{O}_1} |e^{iT \Delta} u_0|^2 dx + \int_{\mathcal{O}_2} |e^{iS \Delta} u_0|^2 dx \geq C^{-1}$ means that every quantum particle visits either $\mathcal{O}_1$ at time $T$, or $\mathcal{O}_2$ at time $S$.\\

	The goal of the article is to provide sufficient conditions of observability of the Schrödinger equation \eqref{eq:SchrodingerObs} based on uncertainty principles for decaying densities.

	\subsection{Main results}
	
Let us introduce the following definition.
\begin{definition}[Thick set with respect to a density]
		Let $\gamma \in (0,1]$, and $\rho: \mathbb{R}^d \to (0, \infty)$ be a continous positive function.
		A set $\mathcal{O} \subset \mathbb{R}^d$ is said to be \emph{$(\gamma, \rho)$-thick} if
\begin{equation}
\label{eq:rhothick}
|\mathcal{O} \cap B(x,  \rho(x))| \ge \gamma \, |B(x,  \rho(x))| \qquad \forall x \in \mathbb{R}^d,
\end{equation}
where $|\cdot|$ stands for the Lebesgue measure in $\R^d$. 
	\end{definition}
A set $\mathcal{O} \subset \mathbb{R}^d$ is said to be thick if there exist $\gamma \in (0,1)$ and $a>0$ such that $\mathcal{O}$ is $(\gamma, a)$-thick.

\medskip

Our first two main results focus on the observability notion defined in \eqref{eq:obs}.
	\begin{theorem}
	\label{tm:mainresult1} 
	Let $\rho : \R^d \to (0,+\infty)$ be a continuous positive function such that 
	\begin{equation}
	\label{eq:cond1}
	\rho(x) = \mathcal{O}(|x|^{-1})\ \text{as}\ |x| \to \infty.
\end{equation}			
There exists $\gamma \in (0,1)$ such that if $\mathcal{O} \subset \R^d$ is a measurable $(\gamma, \rho)$-thick set, then the Schrödinger equation \eqref{eq:SchrodingerObs} is observable in $(0,T)\times\mathcal{O}$ for some time $T>0$. More precisely, there exist $T_0>0$ and $C>0$ such that for every $T \geq T_0$, \eqref{eq:obs} holds.
\end{theorem}	
\begin{theorem}
\label{tm:mainresult2} 
Let $\alpha >1$ and $\rho : \R^d \to (0,+\infty)$ be a continuous positive function such that 
	\begin{equation}
		\label{eq:cond2}
	\rho(x) = \mathcal{O}(|x|^{-\alpha})\ \text{as}\ |x| \to \infty.
\end{equation}	 
There exists $\gamma \in (0,1)$ such that if $\mathcal{O} \subset \R^d$ is a measurable $( \gamma, \rho)$-thick set, then the Schrödinger equation \eqref{eq:SchrodingerObs} is observable in $(0,T)\times\mathcal{O}$ for every time $T>0$. More precisely, there exists $C>0$ such that for every $T>0$,
\begin{equation}\label{equ-obs-schrothm2}
			\|u_0\|_{L^2(\R^d)}^2\le Ce^{\frac{C}{T^2}}\int_0^T\int_{\mathcal{O}}|e^{it\Delta}u_0|^2dxdt\qquad \forall u_0\in L^2(\R^d).
		\end{equation}
\end{theorem}	

\textbf{Bibliographical comments.} In the one-dimensional case, i.e. $d=1$, it has been shown in \cite{MartinPravdaStarov2020} and \cite{HuangWangWang2022} that the free Schrödinger equation \eqref{eq:SchrodingerObs} is observable in $(0,T)\times \mathcal O$ at some time $T>0$ if and only if $\mathcal{O}$ is thick ($\rho(x) \equiv a >0$). The observability at every time $T>0$ has been recently obtained in \cite{SSY23}, always assuming that $\mathcal{O}$ is a thick set. In the multi-dimensional case, i.e. $d \geq 2$, only partial results have been obtained for identifying a necessary and sufficient condition on $\mathcal{O}$ for the observability \eqref{eq:obs} to hold. Let us sum up the known results. It has first been shown in \cite{MartinPravdaStarov2020} that the thickness property is a necessary condition for \eqref{eq:obs}. The other cases concern sufficient condition. From \cite{BJ16}, \cite{Tau22}, one can deduce that \eqref{eq:SchrodingerObs} is observable in $(0,T)\times \mathcal{O}$ for every time $T>0$ if $\mathcal{O}$ is an open set satisfying the so-called Geometric Control Condition (GCC). This condition roughly tells that every infinite line intersects $\mathcal{O}$. In \cite{WangWangZhang2019}, the authors have proved that if $\mathcal{O} \subset \R^d$ is such that $|\mathcal{O}^c| < +\infty$  then \eqref{eq:SchrodingerObs} is observable in $(0,T)\times \mathcal{O}$ for every time $T>0$. In \cite{Wun17}, \cite{Tau22}, the authors prove that if $\mathcal{O} \subset \R^d$ is a periodic open subset of $\R^d$, then \eqref{eq:SchrodingerObs} is observable in $(0,T)\times \mathcal{O}$ for every time $T>0$. A generalization to periodic measurable subset in $\R^2$ has further been obtained in \cite{LBM23}. Let us also mention that for Schrödinger equations with a confining potential, \cite{Pro25} characterizes open sets for which observability holds.

\medskip From the previous discussion, we deduce that our main results from \Cref{tm:mainresult1} and \Cref{tm:mainresult2} only provide sufficient conditions for the observability \eqref{eq:obs} to hold. They are not sharp as the one-dimensional situation tells us. It is however interesting to compare these conditions with the other exhibited conditions in the litterature. First, the thickness condition with respect to $\rho$, with $\rho$ satisfying either \eqref{eq:cond1} or \eqref{eq:cond2} imply the usual thickness condition (constant density). This is in adequation with the fact that the thickness property is a necessary condition for \eqref{eq:obs} to hold. On the other hand, it is interesting to remark that the observation set $\mathcal{O} = \{(x,y) \in \R^2 \ ;\ |x y | > C\}$ is a thick set with respect to $\rho(x) = \min(1, |x|^{-1})$ but $\mathcal{O}$ does not satisfy (GCC), neither $|\mathcal{O}^c| < +\infty$. Conversely, one can also construct a set satisfying both (GCC) and $|\mathcal{O}^c| < +\infty$, but does not satisfy the thickness property \eqref{eq:rhothick} with respect to $\rho(x)$ satisfying \eqref{eq:cond1}. Indeed, take $\mathcal{O} = \R^d \setminus \left(\cup_{k \in \N} B(x_k, |x_k|^{-1/2})\right)$ where $x_k = 2^{k} e_1$.

\bigskip

Our third main result focus on the observability notion at two times points defined in \eqref{eq:obstwopoints}.
	\begin{theorem}
		\label{thm:mainresult3}
		Let $\alpha>0$, $\rho_{\alpha}$ and $\rho_{1/\alpha}$ be two continuous positive functions such that 
\begin{equation}
\rho_{\alpha}(x) = \mathcal{O}(|x|^{-\alpha}),\ \rho_{1/\alpha}(x) = \mathcal{O}(|x^{-1/\alpha}|)\ \text{as}\ |x| \to \infty.
\end{equation}		
For any $T>S>0$,
		there exists $\gamma=\gamma(T-S) \in (0,1)$ such that if $\mathcal{O}_1 \subset \R^d$ is a measurable $( \gamma, \rho_{\alpha})$ thick set and $\mathcal{O}_2 \subset \R^d$ is a measurable $( \gamma, \rho_{1/\alpha})$ thick set, then the Schrödinger equation \eqref{eq:SchrodingerObs} is observable in $(\{T\} \times \mathcal{O}_1) \cup (\{S\} \times \mathcal{O}_2)$. 
	\end{theorem}
	
	\textbf{Bibliographical comments.} The observability inequality at two times points \eqref{eq:obstwopoints} has been first investigated in \cite{WWZZ19}. In this paper, the authors obtain that \eqref{eq:obstwopoints} holds for every measurable sets $\mathcal{O}_1$, $\mathcal{O}_2$, with $|\mathcal O_1^c| < +\infty$,  $|\mathcal O_2^c| < +\infty$. Then, other articles have been investigated this notion for different equations, see for instance \cite{LW21}, \cite{HS21}, \cite{WW22}, \cite{WLH23}, \cite{LW25}.
	
	\medskip
	
	Before continuing, we make the following two remarks.
\begin{remark}
		It is interesting to note that \Cref{tm:mainresult1} could be deduced from \Cref{thm:mainresult3} for arbitrary $T>0$, at the cost that the observation set $\mathcal{O}$ depends on $T$. Indeed, fix $T_0>0$ and choose $\alpha=1$. Then, by \Cref{thm:mainresult3}, there exists 
		$\gamma=\gamma(T_0)\in(0,1)$ and a $(\gamma,\rho_\alpha)$-thick set 
		$\mathcal{O}=\mathcal{O}_1=\mathcal{O}_2$ such that for all times $T \ge S+T_0>T_0$, for every $u_0 \in L^2(\R^d)$, 
		$$
		\|u_0\|_{L^2(\R^d)}^2 \leq C \left(\int_{\mathcal{O}_1} |e^{iT \Delta} u_0|^2 dx + \int_{\mathcal{O}_2} |e^{iS \Delta} u_0|^2 dx \right).
$$
By integrating for $S \in (0,T_0)$, taking $T=S+T_0$, we obtain that 
		$$
		\int_{0}^{T_0} \|u_0\|_{L^2(\R^d)}^2 \leq C \left(\int_{0}^{T_0} \int_{\mathcal{O}_1} |e^{i(S+T_0) \Delta} u_0|^2 dx dS +\int_{0}^{T_0} \int_{\mathcal{O}_2} |e^{iS \Delta} u_0|^2 dx dS \right),
		$$
		so 
		$$\|u_0\|_{L^2(\R^d)}^2 \leq C \int_{0}^{2 T_0} \int_{\mathcal{O}} |e^{it \Delta} u_0|^2 dx dt.$$
		This leads to \eqref{eq:obs}.
	\end{remark}
	
	\begin{remark}
	\label{remark:sharptwotimes1}
	In some sense, \Cref{thm:mainresult3} is sharp in the situation where $\mathcal{O}_1$, $\mathcal{O}_2$ are measurable thick sets with respect to the same density $\rho$. Indeed, in this case, \Cref{thm:mainresult3} asserts that a sufficient condition that guarantees \eqref{eq:obstwopoints} to hold is that $\rho(x) = \mathcal O(|x|^{-1})$ as $|x| \to +\infty$. On the reverse side, as we will see below, see \Cref{remark:sharptwotimes}, if $\rho$ is a continuous positive function such that $|x|\rho(x)  \to 0$ then for every $\gamma \in (0,1)$, one could find a $(\gamma, \rho)$-thick set $\mathcal{O}$ such that the inequality \eqref{eq:obstwopoints} fails for $\mathcal{O} = \mathcal{O}_1 = \mathcal{O}_2$.
	\end{remark}
	
\textbf{Strategy of the proof.} The proof of our first two main results \Cref{tm:mainresult1} and \Cref{tm:mainresult2} consists in proving resolvent estimates. Abstract results, see \Cref{prop:Miller}, \Cref{thm:resolv-obs} below, essentially due to Miller \cite{Mil2005} and Su, Sun, Yuan \cite{SSY23}, are then used to deduce \Cref{tm:mainresult1} and \Cref{tm:mainresult2}. The resolvent estimates are obtained through the use of an uncertainty principle for decaying densities due to Kovrijkine \cite{Kovrizhkin2003}, see \Cref{thm:kov2003} below. The proof of our third main result \Cref{thm:mainresult3} is a rather straightforward consequence of Kovrijkine's uncertainty principle and a result due to Wang, Wang, Zhang and Zhang \cite{WWZZ19} establishing a link between uncertainty principle and two times points observability of the Schrödinger equation, see \Cref{0309-sch-eq-twopoints} below.

\bigskip 

\textbf{Acknowledgements.} The authors are grateful to Chenmin Sun for interesting discussions, notably leading to the extension of one of our main results by the adding of some real-valued bounded potential $V$, see \Cref{sec:V} below.

\section{Main ingredients of the proof}

\subsection{Quantitative uncertainty principles for decaying densities}

In this part, we state a quantitative uncertainty principle due to Kovrijkine \cite{Kovrizhkin2003}, strongly inspired by the one of Shubin, Vakilian and Wolff \cite{ShubinVakilianWolff1998}. This ingredient will be the heart of the proof of our main results \Cref{tm:mainresult1}, \Cref{tm:mainresult2} and \Cref{thm:mainresult3}.

\medskip

For a function $f \in L^1(\R^d) \cap L^2(\R^d)$, the Fourier transform of $f$ is defined by
	\begin{equation}
		\label{eq:fourierf}
		\widehat{f}(\xi) = \int_{\R^d} f(x) e^{-2 i \pi \langle x, \xi\rangle} d x,
	\end{equation}
	that can be classically extended to all $f \in L^2(\R^d)$.
	
	\medskip
	
	The following result holds.

\begin{theorem}[\cite{Kovrizhkin2003}]
		\label{thm:kov2003}
		Let $\rho_1=\rho_1(|x|)$, $\rho_2=\rho_2(|x|)$ be two continuous positive radial functions, such that there exist $C_1, C_2 >0$,
		\begin{equation}
			\label{eq:generalkov}
			\frac{C_1}{\rho_1\left(\frac{C_2}{\rho_2(t)}\right)} \geq t\qquad \forall t \geq 0.
		\end{equation}
		Then there exist $\gamma \in (0,1)$, $C = C(d, \gamma) > 0$ such that for every $(\gamma, \rho_1)$-thick set $\mathcal{O}$ and every $( \gamma, \rho_{2})$-thick set $\Omega$, and every $f \in L^2(\mathbb{R}^d)$, 
		\begin{equation}\label{eq:kov}
			\int_{\mathbb{R}^d} |f(x)|^2 \, dx \le C \left( \int_\mathcal{O} |f(x)|^2 \, dx + \int_\Omega |\widehat{f}(\xi)|^2 \, d\xi \right).
		\end{equation}
	\end{theorem}
A standard choice of the densities is $\rho_1(x) = 	\min(1, |x|^{-\alpha})$ and $\rho_2(x) =  \min(1, |x|^{-1/\alpha})$ for $\alpha >0$. The case $\alpha=1$ has been initiated by \cite{ShubinVakilianWolff1998}. In this latter case, $\mathcal{O}$ and $\Omega$ are thick sets with respect to the same density $\rho(x) = \min(1, |x|^{-1})$. The decaying assumption is also known to be sharp according to

\begin{remark}
		\label{rmk:sharpwolff}
If $\rho(x)$ is such that $|x| \rho(x) \to +\infty$, then for every $\gamma \in (0,1)$, $N \in \N^*$, there exists a $(\gamma, \rho)$-thick set $\mathcal{O}_N$ and a function $f_N \in L^2(\mathbb{R}^d)$ such that
		\begin{equation}\label{eq:wolffsharp}
			\int_{\mathcal{O}_N} |f_N(x)|^2  dx + \int_{\mathcal{O}_N} |\widehat{f_N}(\xi)|^2  d\xi  \leq \frac{\|f_N\|_{L^2(\R^d)}}{N}.
		\end{equation}
		In particular, the estimate \eqref{eq:kov} cannot hold whatever $\gamma \in (0,1)$ is.
	\end{remark}

The equation \eqref{eq:kov} is called a quantitative uncertainty principle. We recall that roughly speaking, uncertainty principle states that a function and its Fourier transform cannot both be arbitrarily well-localized. A well-known (qualitative) ``uncertainty principle'' proved by Benedicks \cite{Ben85} states that if a function $f \in L^2(\R^d)$ is such that $|\mathrm{supp}(f)| < +\infty$ and $|\mathrm{supp}(\widehat{f})| < +\infty$ then $f \equiv 0$. A quantitative form of this uncertainty principle was given by Amrein, Berthier, see \cite{AB77}. A sharp version was obtained by Nazarov \cite{Nazarov1994} in the one-dimensional case. This result was later extended to the multi-dimensional case by Jaming \cite{Jaming2007}. Another related quantitative uncertainty principle is the theorem of Logvinenko and Sereda \cite{LogvinenkoSereda1974}. It expresses that a band-limited function cannot vanish on a set of positive density. A quantitative dependence on the spectral width $b$, the thickness parameter $\gamma$ in Logvinenko, Sereda's theorem was later obtained by Kovrijkine in \cite{Kovrijkine2001}. This result was finally used by Egidi, Veselic \cite{EV18} and independently by Wang, Wang, Zhang, Zhang \cite{WWZZ19} to prove that the observability in $(0,T)\times \mathcal{O}$ of the heat equation holds if and only if $\mathcal{O}$ is a thick set.

\subsection{Abstract results for the observability of the Schrödinger equation}

	The proof of \Cref{tm:mainresult1} relies crucially on the following criterion of resolvent estimates.
	
	\begin{theorem}[{\cite[Theorem 5.1]{Mil2005}}]\label{prop:Miller}
		Assume that the observability estimate \eqref{eq:obs} holds in $(0,T) \times \mathcal{O}$ for the Schrödinger equation \eqref{eq:SchrodingerObs} then there  
		exist  positive constants $M>0$ and $m>0$ such that
		\begin{equation} \label{eq:resolventschro}
			\| f \|_{L^2(\mathbb{R}^d)}^2 \leq M\|(-\Delta-\lambda)f\|_{L^2(\mathbb{R}^d)}^2 +m \| f \|_{L^2(\mathcal{O})}^2, \qquad \forall f \in L^2(\R^d),\ \forall \lambda \in \R.
		\end{equation}
		Conversely, if the observability resolvent estimate \eqref{eq:resolventschro} hold for some constants $M>0$ and $m>0$, 
		then for all $\epsilon>0$, there is a constant $C_\epsilon>0$ such that the observability estimate \eqref{eq:obs} holds in $(0,T) \times \mathcal{O}$ with $C_{obs}=C_\epsilon mT/(T^2-M(\pi^2+\epsilon))$ for every time $T > \sqrt{M(\pi^2+\epsilon)}$.
	\end{theorem}
	
The proof of \Cref{tm:mainresult2} mainly follows the strategy of \cite{SSY23}.	We formulate an abstract result, that roughly says that the thickness condition together with high-frequency resolvent estimates lead to observability of the Schr\"{o}dinger equation.
	
	\begin{theorem}\label{thm:resolv-obs}
		Let $\mathcal{O}\subset\R^d$ be a thick set. Assume there exist $m>0,\,\lambda_0>0$ such that
		\begin{equation}\label{equ-resolvent}
			\|f\|_{L^2(\R^d)}^2\le M(\lambda) \|(-\Delta-\lambda)f\|_{L^2(\R^d)}^2+m\|f\|_{L^2(\mathcal{O})}^2,\qquad \forall f\in L^2(\R^d),\ \forall \lambda \geq \lambda_0,
		\end{equation}\label{equ-limit}
		where $M(\lambda)>0$ is decreasing and satisfies the limit condition
		\begin{equation}
		\label{eq:LimM}
			\lim_{\lambda\to\infty}M(\lambda)=0.
		\end{equation}
	Then, there exists $C=C(d,\mathcal{O})$ such for any $T>0$, we have the observability estimate
		\begin{equation}\label{equ-obs-schro}
			\|u_0\|_{L^2(\R^d)}^2\le Ce^{\frac{C}{T^2}}\int_0^T\int_{\mathcal{O}}|e^{it\Delta}u_0|^2dxdt\qquad \forall u_0\in L^2(\R^d).
		\end{equation}
	\end{theorem}
	
 \Cref{thm:resolv-obs} is not explictly stated in \cite{SSY23}, this is why we decide to give a sketch of the proof following \cite{SSY23} of this result in \Cref{sec:appendixa} below. 
 
\medskip 
 
Let us make some comments on Theorem \ref{thm:resolv-obs}. First, we know from \cite{MartinPravdaStarov2020} that the thickness condition on $\mathcal O$ is necessary for the observability \eqref{equ-obs-schro} to hold. This assumption turns out to be equivalent to the following spectral inequality
		\begin{equation}\label{equ-spec-ball}
			\|f\|_{L^2(\R^d)}^2\le e^{C_{spec}(1+\lambda)}\|f\|_{L^2(\mathcal{O})}^2\;\mbox{ for any } f\in L^2(\R^d) \mbox{ with supp }\hat{f}\subset B(0,\lambda),
		\end{equation}
according to \cite{Kovrijkine2001}. Finally, this assumption is also equivalent to the observability of the heat equation in the Euclidean space $\R^d$, i.e.
\begin{equation}
\label{eq:obsheat}
			\|u_0\|_{L^2(\R^d)}^2\le C e^{\frac{C}{T}}\| e^{t \Delta} u_0\|_{L^2(\mathcal{O})}^2\qquad \forall u_0 \in L^2(\R^d),
\end{equation}
according to \cite{EV18} and \cite{WWZZ19}. In our proof of Theorem \ref{thm:resolv-obs}, we will directly use \eqref{eq:obsheat} instead of the thickness condition to obtain \eqref{equ-obs-schro}.

 \medskip
	
The proof of \Cref{thm:mainresult3} is based on the equivalence between uncertainty principles and observability at two-times points of the Schrödinger equation \eqref{eq:SchrodingerObs}.
	
	\begin{theorem}[\cite{WangWangZhang2019}]\label{0309-sch-eq-twopoints}
		Let $A$ and $B$ be two measurable subsets of $\mathbb R^d$. The following propositions are equivalent.
		
(i) There exists a positive constant $C_1(d,A,B)$ so that for each $f\in L^2(\mathbb R^d)$,
		\begin{equation}\label{eq:uncertaintylemma}
			\int_{\mathbb R^d}  |f(x)|^2 dx
			\leq C_1(d,A,B)
			\left( \int_{A}  |f(x)|^2 dx + \int_{B}  |\widehat f(\xi)|^2d \xi\right).
		\end{equation}
		
(ii) There exists a positive constant $C_2\big(n,A,B\big)$ so that for each $T>0$ and each $u_0\in L^2(\mathbb R^d)$,
		\begin{equation}\label{eq:uncertaintylemmaequiv}
			\int_{\mathbb R^d}  |u_0(x)|^2 dx
			\leq C_2(d,A,B)  \Big(\int_{ A}  |u_0(x)|^2 dx
			+ \int_{2TB}  |e^{iT \Delta} u_0|^2 dx\Big).
		\end{equation}

		Furthermore, when one of the above two propositions holds, the constants $C_1(d,A,B)$ and $C_2(d,A,B)$ can be chosen as the same  number.

	\end{theorem}

\section{Proofs of the main results}

This part is devoted to the proof of our main results.
	
	\subsection{Resolvent estimates}
	
	In this part, we derive resolvent estimates.
	
	\medskip
	
	In the following, we denote by $C_d'$ a universal positive constant depending only on $d$ and by $\omega_d$ the volume of the unit ball in $\R^d$.
	
	\medskip
	
	The goal of this part is to establish the following result.
	
	\begin{proposition}\label{prop-resolvent}
		Let $\alpha \ge 1$, $\rho_{\alpha}(x)=\min\{1,|x|^{-\alpha}\}$. There exist $\gamma\in(0,1)$, $\lambda_0=(\frac{3^{2+1/\alpha}C_d'}{(1-\gamma)\omega_d}+1)^2$, $c=c(d,\gamma,\alpha)>0$ such that for any $\mathcal{O}\subset\R^d$ measurable $(\gamma,\rho_\alpha)$-thick set, for any $\lambda\ge\lambda_0$, 
		\begin{equation*}
			c\|f\|_{L^2(\R^d)}^2\le M(\lambda) \|(-\Delta-\lambda)f\|_{L^2(\R^d)}^2+\|f\|_{L^2(\mathcal{O})}^2\qquad \forall f\in L^2(\R^d),\ \forall \lambda\ge\lambda_0,
		\end{equation*}
where
\begin{equation}
M(\lambda)=\lambda^{-(1-1/\alpha)}.
\end{equation}
	\end{proposition}
	
	The proof of \Cref{prop-resolvent} is postponed to the end of this part.\medskip
	
	In the following, for $\beta > 0$ and $\lambda>0$, define
	\[
	A_{\lambda,\beta} = \{ \xi \in \mathbb{R}^d : \, \big||\xi| - \sqrt{\lambda}\big| \le \lambda^{-\beta/2}\}.
	\]
	Now, we give the thinness property of $A_{\lambda,\beta}$.
	
	\begin{lemma}\label{lem:thinness}
		For every $\varepsilon>0$, we define $L=\frac{3^{1+\beta}C_d'}{\varepsilon \omega_d}$.
		Take
		\[
		\tilde{\rho}(x)=\min\{L,L^{1+\beta}|x|^{-\beta}\}.
		\]
		Then, there exists ${\lambda}_0= (2L+1)^2 >0$ such that,
		\begin{equation*}
			|A_{\lambda,\beta} \cap B(x,\tilde{\rho}(x))| \le \varepsilon |B(x,\tilde{\rho}(x))|
			\qquad \forall x \in \mathbb{R}^d, \,\ \forall \lambda \geq {\lambda}_0.
		\end{equation*}
	\end{lemma}

	\begin{proof}
		We set 
		\[
		a := \lambda^{1/2}-\lambda^{-\beta/2}, 
		\qquad 
		b := \lambda^{1/2}+\lambda^{-\beta/2},
		\]
		so that $A_{\lambda,\beta} = \{\xi \in \mathbb{R}^d : a \le |\xi| \le b\}$.
		
		\medskip
		
\textit{Step 1: Case $|x|\le L$.} For all $\beta>0$,
		$\tilde{\rho}(x) = \min(L, L^{1+\beta}|x|^{-\beta})=L$, so $B(x,\tilde{\rho})\subset B(0,2L)$.
		Since $\lambda\ge{\lambda}_0= (2L+1)^2$, we have $a\ge {\lambda}_0^{1/2}-{\lambda}_0^{-\beta/2}>{\lambda}_0^{1/2}-1=2L$. Thus, $A_{\lambda,\beta}\subset B(0,2L)^c$ and $A_{\lambda,\beta}\cap B(x,\tilde{\rho}(x))=\emptyset$. The claim is trivial.
		
				\medskip
		
	\textit{Step 2: Case $|x| > L$.}
		Here $\tilde{\rho}(x) = L^{1+\beta}|x|^{-\beta} < L$.  
		Denote
		\[
		E := A_{\lambda,\beta} \cap B(x,\tilde\rho).
		\]
		
		We write $\xi = r\omega$, where $r = |\xi| \ge 0$ and $\omega \in \mathbb{S}^{d-1}$. 
		The volume element is $d\xi = r^{d-1} dr \, d\sigma(\omega)$, where $d\sigma$ is the surface measure on $\mathbb{S}^{d-1}$.
		Thus,
		\[
		|E| = \int_a^b \int_{\mathbb{S}^{d-1}} \mathbf{1}_{\{|r\omega - x| < \tilde\rho\}} \, d\sigma(\omega)\, r^{d-1}\,dr.
		\]
		For each $r \in [a,b]$ define
		\[
		\Omega_r := \{\omega \in \mathbb{S}^{d-1} : |r\omega - x| < \tilde\rho\}.
		\]
		Then
		\[
		|E| = \int_a^b \sigma(\Omega_r)\, r^{d-1}\,dr.
		\]
		
		\medskip
		
		Fix $r \in [a,b]$ and suppose $\Omega_r \neq \emptyset$. 
		Pick $\omega_0 \in \Omega_r$. 
		Then for any $\omega \in \Omega_r$,
		\[
		|\omega - \omega_0| 
		\le \frac{|r\omega - r\omega_0|}{r}
		= \frac{|(r\omega - x) - (r\omega_0 - x)|}{r}
		\le \frac{|r\omega - x| + |r\omega_0 - x|}{r}
		< \frac{2\tilde\rho}{r}.
		\]
		Hence $\Omega_r \subset \{\omega \in \mathbb{S}^{d-1} : |\omega - \omega_0| < 2\tilde\rho/r\}$. For small $\delta > 0$ the surface measure of a spherical cap of Euclidean radius $\delta$ on $\mathbb{S}^{d-1}$ is bounded by $C_d \delta^{d-1}$ for some constant $C_d$ depending only on $d$. 
		Therefore,
		\[
		\sigma(\Omega_r) \le C_d \Big(\frac{2\tilde\rho}{r}\Big)^{d-1}
		= C_d' \frac{\tilde\rho^{d-1}}{r^{d-1}},
		\]
		where $C_d' = C_d 2^{d-1}$.
		
\medskip
		
		Substitute this bound into the integral
		\[
		|E| 
		\le \int_a^b C_d' \frac{\tilde\rho^{d-1}}{r^{d-1}}\,r^{d-1}\,dr
		= 2C_d'L^{(1+\beta)(d-1)}\lambda^{-\beta/2}|x|^{-\beta(d-1)}.
		\]
		
		\medskip

		Since the intersection is nonempty, there exists $\xi \in E$ satisfying $|\xi-x|\le \tilde\rho(x)<L$ and $|\xi|\le b<\lambda^{1/2}+1$.
		Then,
		\[
		|x|\le |\xi-x|+|\xi|<L+\lambda^{1/2}+1<3L\lambda^{1/2}.
		\]
		Thus,
		\[
		\frac{|A_{\lambda,\beta} \cap B(x,\tilde{\rho})|}{|B(x,\tilde{\rho})|}=\frac{|E|}{\omega_d(L^{1+\beta}|x|^{-\beta})^d}
		\le \frac{2C_d'}{\omega_d L^{1+\beta}}\lambda^{-\beta/2}|x|^\beta\le \frac{3^{1+\beta}C_d'}{\omega_d}L^{-1}=\varepsilon.
		\]
		This completes the proof.
	\end{proof}
	
		Then we give the proof of Proposition \ref{prop-resolvent}.
	
	\begin{proof}[Proof of Proposition \ref{prop-resolvent}]
		We borrow some ideas from \cite{Gre2020}. Note that, following the proof of \cite[Lemma 1]{Gre2020}, one has the algebraic argument.
		Let $s>0$, there exists $c_s>0$, such that 
		\begin{equation}\label{equ-algebra}
			|\tau^s-\lambda|\ge c_s D(\lambda+D^s)^{1-1/s},
		\end{equation}
		for all $\tau,\lambda\ge 0$ in the region $|\tau-\lambda^{1/s}|>D$.

		By \Cref{lem:thinness}, for every $\gamma\in(0,1)$, $\lambda\ge\lambda_0$, $A_{\lambda,1/\alpha}^c$ is a $(\gamma,\tilde\rho_{1/\alpha})$-thick set, where $\tilde{\rho}_{1/\alpha}=\min\{L,L^{1+1/\alpha}|x|^{-1/\alpha}\}$, $L=\frac{3^{1+1/\alpha}C_d'}{(1-\gamma) \omega_d}$. Then, by calculation, we have
		\begin{equation*}
			\frac{L}{\tilde{\rho}_{1/\alpha}\big(\frac{L}{\rho_\alpha(t)}\big)}\ge t,\;\quad\mbox{for all }t\ge0.
		\end{equation*}
		From Kovrijkine's uncertainty principle stated in \Cref{thm:kov2003}, we obtain that 
		for any $\lambda\ge \lambda_{0}$, $g\in L^2(\R^d)$ with $\mbox{supp }\hat{g}\subset A_{\lambda,1/\alpha}$, 
		\begin{equation}\label{equ-spect}
			\int_{\R^d}|g|^2dx\le C(d,\gamma,L)\int_{\mathcal{O}}|g(x)|^2dx.
		\end{equation}
		
		Denote by $P_\lambda$ the projection $P_\lambda f=\mathcal{F}^{-1}(\mathbbm{1}_{A_{\lambda,1/\alpha}}\mathcal{F}(f))$. Applying \eqref{equ-algebra} with $\tau=|\xi|,\,s=2,\,D=\lambda^{-1/(2\alpha)}$, we obtain that for every $f\in L^2(\R^d)$,
		\begin{align}\label{equ-highfreq}
			\|(-\Delta-\lambda)f\|_{L^2(\R^d)}^2&=\int_{\R^d}\big||\xi|^2-\lambda\big|^2|\hat{f}(\xi)|^2 d\xi\notag\\
			&\ge \int_{A_{\lambda,1/\alpha}^c}\big||\xi|^2-\lambda\big|^2|\hat{f}(\xi)|^2 d\xi\notag\\
			&\ge c_2^2\lambda^{-1/\alpha}(\lambda+\lambda^{-1/\alpha})\int_{A_{\lambda,1/\alpha}^c}|\hat{f}(\xi)|^2 d\xi\notag\\
			&=c_2^2 \lambda^{1-1/\alpha}(1+\lambda^{-(1+1/\alpha)})\|(I-P_\lambda)f\|_{L^2(\R^d)}^2.
		\end{align}
		Then, together \eqref{equ-spect} with \eqref{equ-highfreq}, for any $\lambda\ge \lambda_{0}$ and every $f\in L^2(\R^d)$, we have
		\begin{align*}
			\|f\|_{L^2(\R^d)}^2&=\|P_\lambda f\|_{L^2(\R^d)}^2+\|(I-P_\lambda) f\|_{L^2(\R^d)}^2\\
			&\le C \|P_\lambda f\|_{L^2(\mathcal{O})}^2+\|(I-P_\lambda) f\|_{L^2(\R^d)}^2\\
			&\le 2C \|f\|_{L^2(\mathcal{O})}^2+(2C+1)\|(I-P_\lambda) f\|_{L^2(\R^d)}^2\\
			&\le 2C \|f\|_{L^2(\mathcal{O})}^2+(2C+1)\frac{\lambda^{-(1-1/\alpha)}}{c_2^{2}(1+\lambda^{-(1+1/\alpha)})}	\|(-\Delta-\lambda)f\|_{L^2(\R^d)}^2\\
			&\le 2C \|f\|_{L^2(\mathcal{O})}^2+(2C+1)c_2^{-2}\lambda^{-(1-1/\alpha)}	\|(-\Delta-\lambda)f\|_{L^2(\R^d)}^2.
		\end{align*}
		Thus, the proof is complete.
	\end{proof}

	\subsection{Proofs of the main results} 
	
	In this part, we prove the main results.
	
	\medskip
	
	We start with a technical lemma strongly inspired by \cite[Lemma 5.6]{MPS23}.
	\begin{lemma}\label{lem:comparaisondensity}
		Let $\gamma>0$ and $\rho_1, \rho_2$ be two continuous positive functions satisfying
		\begin{equation}
			\label{eq:comparaisondensity}
			0<\rho_1(x)\le \rho_2(x),\qquad \forall x \in \R^d.
		\end{equation}
		Assume that $\mathcal{O} \subset\mathbb{R}^d$ is a measurable
		subset which is $(\gamma, \rho_1)$-thick.
		Then $\mathcal{O}$ is $(\gamma/9^{d},3 \rho_2)$-thick.
	\end{lemma}
	
	\begin{proof}
		By hypothesis we have
		\begin{equation}
			\label{eq:thickrho1}
			|\omega\cap B(x,\rho_1(x))|\ge \gamma\,|B(x,\rho_1(x))|,\qquad \forall x\in\mathbb{R}^d.
		\end{equation}
		
		Fix $x\in\mathbb{R}^d$. We first observe the inclusion
		\[
		B\big(x,\rho_2(x)\big)
		\subset
		\bigcup_{\substack{y\in B(x,\rho_2(x))\\ \rho_1(y)\le 2\rho_2(x)}}
		B\big(y,\rho_1(y)\big).
		\]
		Indeed, if $z\in B(x,\rho_2(x))$ and $\rho_1(z)>2\rho_2(x)$,
		consider the continuous function $f(t)=\rho_1(t z+(1-t)x)$ for $t\in[0,1]$. We have
		$f(0)=\rho_1(x)\le \rho_2(x)$ and $f(1)=\rho_1(z)>2\rho_2(x)$,
		so by continuity there exists $t_0\in(0,1)$ such that for $y:=t_0 z+(1-t_0)x$ one has
		$\rho_1(y)=2\rho_2(x)$ and $z\in B(y,\rho_1(y))$. Moreover $y\in B(x,\rho_2(x))$,
		hence the claim.
		
		By compactness there exists a finite family $(x_i)_{0\le i\le N}\subset B(x,\rho_2(x))$
		such that
		\begin{equation}\label{eq:finite-cover}
			B\big(x,\rho_2(x)\big)
			\subset \bigcup_{i=0}^N B\big(x_i,\rho_1(x_i)\big),
			\quad\text{and}\quad \rho(x_i)\le 2\rho_2(x)\ \text{ for all }i.
		\end{equation}
		
		Apply the Vitali covering lemma to the finite family $\{B(x_i,\rho_1(x_i))\}_{i=0}^N$: there exists
		a subset $S\subset\{0,\dots,N\}$ such that the balls $(B(x_i,\rho_1(x_i)))_{i\in S}$ are pairwise disjoint
		and
		\[
		\bigcup_{i=0}^N B(x_i,\rho_1(x_i)) \subset \bigcup_{i\in S} B(x_i,3\rho_1(x_i)).
		\]
		Using \eqref{eq:thickrho1}, \eqref{eq:finite-cover} and the disjointness, we get
		\begin{align*}
			|\omega\cap B(x,3\rho_2(x))|
			&\ge \sum_{i\in S} |\omega\cap B(x_i,\rho_1(x_i))|
			\ge \gamma\sum_{i\in S} |B(x_i,\rho_1(x_i))|\\
			&= \frac{\gamma}{3^d}\sum_{i\in S} |B(x_i,3\rho_1(x_i))|
			\ge \frac{\gamma}{3^d}\Big|\bigcup_{i\in S} B(x_i,3\rho_1(x_i))\Big|\\
			&\ge \frac{\gamma}{3^d}\,|B(x,\rho_2(x))|
			=\frac{\gamma}{9^{d}}\,|B(x,3\rho_2(x))|.
		\end{align*}
		The chain of inclusions/estimates above uses that $B(x,\rho_2(x))\subset \bigcup_{i\in S} B(x_i,3\rho_1(x_i))$
		and the volume scaling $|B(\cdot,3r)|=3^d|B(\cdot,r)|$.
		
		The conclusion then follows.
	\end{proof}

	With this at hand, we then prove \Cref{tm:mainresult1}.
	
	\begin{proof}[Proof of \Cref{tm:mainresult1}]
		Noticing the fact $\rho_{\alpha}(x)\le 1$ for all $x\in\R^d$ and $\alpha>0$, we derive that $\mathcal{O}$ is a $(\gamma/9^d,3)$-thick set from \Cref{lem:comparaisondensity}.
		
		\medskip
		
		\emph{Step 1: Reduction to $\rho_\alpha(x)=\min(1, |x|^{-\alpha})$.} From Lemma \ref{lem:comparaisondensity}, it is not restrictive to assume that $\rho_\alpha(x) = \min(1, |x|^{-\alpha})$. Indeed, by hypothesis, we have that there exists $c>0$ such that
		\begin{equation}
			0 < c \rho_\alpha(x) \leq \frac{\min(1, |x|^{-\alpha})}{3}.
		\end{equation}
		So \eqref{eq:comparaisondensity} is fulfilled, then the $(\gamma, \rho_\alpha)$-thickness of $\mathcal{O}$ directly implies the $(\gamma 9^{-d}, \rho_\alpha(x) =  \min(1, |x|^{-\alpha}))$ thickness of $\mathcal{O}$.
		
				\medskip
		
		\emph{Step 2: Proof of the resolvent estimates for every $\lambda$.}
%
		 By \Cref{prop-resolvent}, there exist $C,m,\lambda_{0}>0$, depending only on $d,\gamma,\alpha$, such that for every $f\in L^2(\R^d)$ and any $\lambda\ge\lambda_0$, we have 
		\begin{equation}\label{equ-resolv-large}
			\|f\|_{L^2(\R^d)}^2\le C\|(-\Delta-\lambda)f\|_{L^2(\R^d)}^2+m\|f\|_{L^2(\mathcal{O})}^2.
		\end{equation}
		For $\lambda<\lambda_0$, 
		if $|\xi|^2<\lambda_0+1$, we get $||\xi|^2-\lambda_0|<2\lambda_0+1$, and 
		\begin{align}\label{equ-lowfreq}
			&\quad\int_{|\xi|^2< \lambda_0+1}||\xi|^2-\lambda_0|^2|\hat{f}|^2d\xi\le (2\lambda_0+1)^2\int_{|\xi|^2< \lambda_0+1}|\hat{f}|^2d\xi\notag\\
			&=(2\lambda_0+1)^2\int_{\R^d}|\mathcal{F}^{-1}\chi_{\{|\xi|^2<\lambda_0+1\}}\mathcal{F}(f)|^2dx\notag\\
			&\le e^{C\lambda_0^{1/2}}\int_{\mathcal{O}}|\mathcal{F}^{-1}\chi_{\{|\xi|^2<\lambda_0+1\}}\mathcal{F}(f)|^2dx\notag\\
			&\le 2e^{C\lambda_0^{1/2}}\big(\int_{\mathcal{O}}|f|^2dx+\int_{\R^d}|\mathcal{F}^{-1}\chi_{\{|\xi|^2\ge\lambda_0+1\}}\mathcal{F}(f)|^2dx\big)\notag\\
			&\le 2e^{C\lambda_0^{1/2}}\big(\int_{\mathcal{O}}|f|^2dx+\int_{|\xi|^2\ge\lambda_0+1}||\xi|^2-\lambda|^2|\hat{f}|^2d\xi\big).
		\end{align}
		In \eqref{equ-lowfreq}, the second inequality follows from Theorem 1 in \cite{Kovrijkine2000} and the fact $\mathcal{O}$ is a thick set; the last inequality follows from $||\xi|^2-\lambda|=|\xi|^2-\lambda\ge\lambda_{0}+1-\lambda>1$ whenever $|\xi|^2 \ge \lambda_0 + 1$ and $\lambda < \lambda_0$.
		Note that $||\xi|^2-\lambda_0|<||\xi|^2-\lambda|$
		if $|\xi|^2\ge\lambda_0>\lambda$. By using \eqref{equ-resolv-large} with $\lambda=\lambda_0$ and \eqref{equ-lowfreq}, we then derive 
		\begin{align*}
			\|f\|_{L^2(\R^d)}^2&\le C\|(-\Delta-\lambda_0)f\|_{L^2(\R^d)}^2+m\|f\|_{L^2(\mathcal{O})}^2\\
			&=C\int_{\R^d}||\xi|^2-\lambda_0|^2|\hat{f}|^2d\xi+m\|f\|_{L^2(\mathcal{O})}^2\\
			&\le C\Big(\int_{|\xi|^2\ge\lambda_0}||\xi|^2-\lambda_0|^2|\hat{f}|^2d\xi+ \int_{|\xi|^2<\lambda_0+1}||\xi|^2-\lambda_0|^2|\hat{f}|^2d\xi\Big)+m\|f\|_{L^2(\mathcal{O})}^2\\
			&\le \big(2Ce^{C\lambda_0^{1/2}}+C\big)\int_{|\xi|^2\ge\lambda_0}||\xi|^2-\lambda|^2|\hat{f}|^2d\xi+ \big(2Ce^{C\lambda_0^{1/2}}+m\big)\|f\|_{L^2(\mathcal{O})}^2\\
			&\le \big(2Ce^{C\lambda_0^{1/2}}+C\big)\|(-\Delta-\lambda_0)f\|_{L^2(\R^d)}^2+ \big(2Ce^{C\lambda_0^{1/2}}+m\big)\|f\|_{L^2(\mathcal{O})}^2.
		\end{align*}
		Therefore, there exist $C,m>0$, depending only on $d,\gamma,\alpha$, such that \eqref{equ-resolv-large} holds for all $\lambda\in\R$. By the abstract result of \Cref{prop:Miller}, we obtain the observability estimate \eqref{eq:obs} for $T>\sqrt{C(\pi^2+\epsilon)}$ with $\epsilon>0$.
	\end{proof}
	
	Now we give the proof of Theorem \ref{tm:mainresult2}.
	\begin{proof}[Proof of Theorem \ref{tm:mainresult2}]
	
		Since $\alpha>1$, it follows from \Cref{prop-resolvent} that there exist $c,\,\lambda_0>0$, depending only on $d,\gamma,\alpha$, such that for any $\lambda\ge\lambda_0$, 
		 \begin{equation*}
		 	c\|f\|_{L^2(\R^d)}^2\le M(\lambda) \|(-\Delta-\lambda)f\|_{L^2(\R^d)}^2+\|f\|_{L^2(\mathcal{O})}^2,\,\mbox{for all }f\in L^2(\R^d),
		 \end{equation*}
		 where $M(\lambda)= C\lambda^{-(1-1/\alpha)}$. Then, since $1-1/\alpha>0$, we have $M(\lambda)$ is decreasing and $M(\lambda)\to 0$ as $\lambda\to\infty$. On the other hand, note that $\mathcal{O}$ is a $(\gamma/{9^d},3)$-thick set. Therefore, applying the abstract result \Cref{thm:resolv-obs}, we conclude the proof.
	\end{proof}

We finally give the proof of Theorem \ref{thm:mainresult3}.

	\begin{proof}[Proof of Theorem \ref{thm:mainresult3}]
		By Theorem \ref{thm:kov2003}, fix $\gamma$ such that for every $(\gamma, \rho_\alpha)$-thick set $A$ and every $( \gamma, \rho_{1/\alpha})$-thick set $B$, and every $f \in L^2(\mathbb{R}^d)$, 
		\begin{equation}\label{eq:kovappBis}
			\int_{\mathbb{R}^d} |f(x)|^2 \, dx \le C \left( \int_A |f(x)|^2 \, dx + \int_B |\widehat{f}(\xi)|^2 \, d\xi \right).
		\end{equation}
		For any $r>0$, the set $r\mathcal{O}_2$ is $(\gamma,\tilde{\rho}_{1/\alpha})$-thick in $\mathbb{R}^d$, where
		$\tilde{\rho}_{1/\alpha}=\min\{r,r^{1+1/\alpha}|x|^{-1/\alpha}\}$. In fact, by the definition of $(\gamma,\rho_{1/\alpha})$-thickness,
		\begin{align*}
			|r\mathcal{O}_2\cap B(x,\tilde{\rho}_{1/\alpha})|=r^d|\mathcal{O}_2\cap B(\frac{x}{r},\frac{\tilde{\rho}_{1/\alpha}(x)}{r})|\ge r^d\gamma|B(\frac{x}{r},\frac{\tilde{\rho}_{1/\alpha}(x)}{r})|=\gamma|B(x,\tilde{\rho}_{1/\alpha})|.
		\end{align*}
		For any $T>S>0$, taking $r=\frac{1}{2(T-S)}$, we have that $A = \mathcal{O}_1$ is $(\gamma, \rho_\alpha)$-thick, and $B= \frac{\mathcal{O}_2}{2(T-S)}=r\mathcal{O}_2$ is $( \gamma, \tilde\rho_{1/\alpha})$-thick.
		A direct computation yields
		\[
		\frac{r}{\tilde{\rho}_{1/\alpha}\big(\frac{r}{\rho_{\alpha}(t)}\big)}\ge t, \quad\mbox{ for all }t\ge 0.
		\]
		Using Theorem \ref{thm:kov2003} again, there exist $\gamma=\gamma(r)\in(0,1),\, C=C(r)>0$ such that
		\begin{equation}\label{eq:kovappBisTer}
			\int_{\mathbb{R}^d} |f(x)|^2 \, dx \le C \left( \int_{\mathcal{O}_1} |f(x)|^2 \, dx + \int_{r{\mathcal{O}}_2} |\widehat{f}(\xi)|^2 \, d\xi \right).
		\end{equation}
		So from Theorem \ref{0309-sch-eq-twopoints} we get (apply at time $T-S>0$),
		\begin{equation}\label{eq:uncertaintylemmaequivapp}
			\int_{\mathbb R^d}  |u_0(x)|^2 dx
			\leq C  \Big(\int_{\mathcal{O}_1}  |u_0(x)|^2 dx
			+ \int_{{\mathcal{O}}_2}  |e^{i(T-S) \Delta} u_0|^2 dx\Big).
		\end{equation}
		Then, we have 
		\begin{equation}\label{eq:uncertaintylemmaequivappBis}
			\int_{\mathbb R^d}  |e^{iS \Delta} u_0|^2 dx
			\leq C  \Big(\int_{\mathcal{O}_1}  |e^{iS \Delta} u_0|^2 dx
			+ \int_{{\mathcal{O}}_2}  |e^{i(T-S) \Delta} e^{iS \Delta} u_0|^2 dx\Big).
		\end{equation}
		Because of the $L^2$-conversation law of the Schr\"{o}dinger equation, the above leads to the inequality \eqref{eq:obstwopoints}. This ends the proof of this theorem.
	\end{proof}
	
	\begin{remark}\label{remark:sharptwotimes}
	To end this part, let us finally mention that the justification of \Cref{remark:sharptwotimes1} directly comes from the proof of of \Cref{thm:mainresult3} and \Cref{rmk:sharpwolff}.
	\end{remark}

	\section{Conclusion and extensions}
	
	In this paper, we derived sufficient conditions for the observability of the Schrödinger equation \eqref{eq:SchrodingerObs}, based on uncertainty principles with respect to decaying densities. These sufficient conditions turn out to be non-sharp according to the one-dimensional case, where it has been proved that the necessary and sufficient condition for the observability of \eqref{eq:SchrodingerObs} is the usual thickness property. In this part, we discuss other situations where the methods of the paper can be applied to prove new results.
	
	\subsection{On the fractional Schrödinger equation}
	
	For $d\geq 1$, $s>1/2$, we are interested in the fractional free Schrödinger equation
	\begin{equation}
		\label{eq:SchrodingerObsfrac}
		\left\{
		\begin{array}{ll}
			i  \partial_t u  - (-  \Delta)^s u = 0 & \text{ in }  (0,+\infty) \times \R^d, \\
			u(0, \cdot) = u_0 & \text{ in } \R^d,
		\end{array}
		\right.
	\end{equation}
	where $u_0 \in L^2(\R^d)$.
	
	\medskip

The observability notion in $(0,T)\times\mathcal O$ is defined as follows
	\be{eq:obsfrac}
			\|u_0\|_{L^2(\R^d)}^2 \leq C \int_0^T \int_{\mathcal{O}} |e^{-it (-\Delta)^s } u_0|^2 dx dt\qquad \forall u_0 \in L^2(\R^d).
			\ee
	
	We have the following generalizations of \Cref{tm:mainresult1} and \Cref{tm:mainresult2}.
	
	\begin{theorem}
	\label{tm:mainresult1Frac} 
	Let $\alpha = \frac{1}{2s-1}$, $\rho : \R^d \to (0,+\infty)$ be a continuous positive function such that 
	\begin{equation}
	\label{eq:cond1Frac}
	\rho(x) = \mathcal{O}(|x|^{-\alpha})\ \text{as}\ |x| \to \infty.
\end{equation}			
There exists $\gamma \in (0,1)$ such that if $\mathcal{O} \subset \R^d$ is a measurable $(\gamma, \rho)$-thick set, then the Schrödinger equation \eqref{eq:SchrodingerObsfrac} is observable in $(0,T)\times\mathcal{O}$ for some time $T>0$. More precisely, there exist $T_0>0$ and $C>0$ such that for every $T \geq T_0$, \eqref{eq:obsfrac} holds.
\end{theorem}	
\begin{theorem}
\label{tm:mainresult2Frac} 
Let $\alpha >\frac{1}{2s-1}$ and $\rho : \R^d \to (0,+\infty)$ be a continuous positive function such that 
	\begin{equation}
		\label{eq:cond2Frac}
	\rho(x) = \mathcal{O}(|x|^{-\alpha})\ \text{as}\ |x| \to \infty.
\end{equation}	 
There exists $\gamma \in (0,1)$ such that if $\mathcal{O} \subset \R^d$ is a measurable $( \gamma, \rho)$-thick set, then the Schrödinger equation \eqref{eq:SchrodingerObsfrac} is observable in $(0,T)\times\mathcal{O}$ for every time $T>0$. More precisely, for every $T>0$, there exists $C>0$ such that \eqref{eq:obsfrac} holds.	
\end{theorem}	
We deduce from the two above results that as $s \to +\infty$, $\alpha \to 0$, so the required thickness assumption on $\mathcal{O}$ becomes weaker.

\medskip

The main ingredient of the proof is based on the following resolvent estimates.

\begin{proposition}\label{prop-resolventFrac}
		Let $\alpha \ge \frac{1}{2s-1}$, $\rho_{\alpha}(x)=\min\{1,|x|^{-\alpha}\}$. There exist $\gamma\in(0,1)$, $\lambda_{0,s}=(\frac{3^{2+1/\alpha}C_d'}{(1-\gamma)\omega_d}+1)^{2s}$, $c=c(d,\gamma,\alpha)>0$ such that for any $\mathcal{O}\subset\R^d$ measurable $(\gamma,\rho_\alpha)$-thick set, for any $\lambda\ge\lambda_0$, 
		\begin{equation*}
			c\|f\|_{L^2(\R^d)}^2\le M(\lambda) \|((-\Delta)^s-\lambda)f\|_{L^2(\R^d)}^2+\|f\|_{L^2(\mathcal{O})}^2.
		\end{equation*}
where
\begin{equation}
M(\lambda)=\lambda^{-(2-1/s-1/(\alpha s))}.
\end{equation}
	\end{proposition}
	
A way to come back to \Cref{prop-resolvent} is first to consider
		\[
		A_{\lambda,\beta,s}=\{\xi\in\R^d:\,||\xi|-\lambda^{\frac1{2s}}|\le \lambda^{-\beta/(2s)}\}.
		\]
Then $A_{\lambda,\beta,s}=A_{\lambda^{1/s},\beta}$.
		By \Cref{lem:thinness}, for every $\varepsilon>0$, there is $\lambda_{0,s}=(\lambda_0)^s$ such that $A_{\lambda,\beta,s}$ is a $(\varepsilon,\tilde\rho_{\beta})$-thin set. Here, $\lambda_{0},\,\tilde\rho_{\beta}=\tilde\rho$ are given in \Cref{lem:thinness}. The rest of the proof is analogous to \Cref{prop-resolvent}, and is left to the reader.

\subsection{On the Schrödinger equation with a real-valued bounded potential}
\label{sec:V}

Let $V \in L^{\infty}(\R^d;\R)$, let us consider the Schrödinger equation
	\begin{equation}
		\label{eq:SchrodingerObsV}
		\left\{
		\begin{array}{ll}
			i  \partial_t u  +  \Delta u - V u = 0 & \text{ in }  (0,+\infty) \times \R^d, \\
			u(0, \cdot) = u_0 & \text{ in } \R^d,
		\end{array}
		\right.
	\end{equation}
	where $u_0 \in L^2(\R^d)$. It is a well-known fact that $i (\Delta-V)$ generates a continuous unitary semigroup in $L^2(\R^d)$ denoted by $(e^{i t (\Delta-V)})_{t \in \R}$. 
	
	\medskip 
	
	We have the following generalization of \Cref{tm:mainresult2}.
	
	\begin{theorem}
\label{tm:mainresult2Bis} 
Let $\alpha >1$ and $\rho : \R^d \to (0,+\infty)$ be a continuous positive function such that 
	\begin{equation}
		\label{eq:cond2Bis}
	\rho(x) = \mathcal{O}(|x|^{-\alpha})\ \text{as}\ |x| \to \infty.
\end{equation}	 
There exists $\gamma \in (0,1)$ such that if $\mathcal{O} \subset \R^d$ is a measurable $( \gamma, \rho)$-thick set, then the Schrödinger equation \eqref{eq:SchrodingerObsV} is observable in $(0,T)\times\mathcal{O}$ for every time $T>0$. More precisely, there exists $C>0$ such that for every $T>0$,
$$\|u_0\|_{L^2(\R^d)}^2 \leq C e^{C/T^2} \int_0^T \int_{\mathcal{O}} |e^{it (\Delta - V) } u_0|^2 dx dt\qquad \forall u_0 \in L^2(\R^d).$$
\end{theorem}	
	
	A direct application of \Cref{prop-resolvent} in the case $\alpha >1$, observing that $M(\lambda) \to 0$ as $\lambda \to +\infty$ enables to prove
\begin{proposition}\label{prop-resolventV}
		Let $\alpha > 1$, $\rho_{\alpha}(x)=\min\{1,|x|^{-\alpha}\}$. There exist $\gamma\in(0,1)$, $\lambda_0=\lambda_0(d, \gamma, \alpha, \|V\|_{\infty})$, $c=c(d,\gamma,\alpha,|V\|_{\infty})>0$ such that for any $\mathcal{O}\subset\R^d$ measurable $(\gamma,\rho_\alpha)$-thick set, for any $\lambda\ge\lambda_0$, 
		\begin{equation*}
			c\|f\|_{L^2(\R^d)}^2\le M(\lambda) \|(-\Delta-\lambda+V)f\|_{L^2(\R^d)}^2+\|f\|_{L^2(\mathcal{O})}^2\qquad \forall f\in L^2(\R^d),\ \forall \lambda\ge\lambda_0
		\end{equation*}
where
\begin{equation}
M(\lambda)=\lambda^{-(1-1/\alpha)}.
\end{equation}
	\end{proposition}
	This corresponds roughly speaking to a high frequency resolvent estimate. One still needs to prove a low frequency resolvent estimate. Note that previously, it was based on the spectral estimate of Kovrijkine's result from \cite{Kovrijkine2001}. In our case of a perturbation by a potential, one cannot use this result. We rather use the recent spectral estimate of \cite[Theorem 2.12]{LM25} for the operator $-\Delta + V$, when $\mathcal{O}$ is a thick set, that directly gives the observability of the associated parabolic equation. The remaining details of the proof are then left to the reader.
	
	\appendix
	
	\section{Proof of \Cref{thm:resolv-obs}}
	\label{sec:appendixa}
	
	The goal of this part is to prove \Cref{thm:resolv-obs}. Our argument is entirely based on the strategy of \cite{SSY23}, which treats the one-dimensional case. We adapt this strategy to the higher-dimensional setting. For completeness, we provide a sketch of the proof and describe its main steps.
	
	\begin{proof}

		\emph{Step 1: Observability for the inhomogeneous backward heat equation.}
		First, the thickness condition implies the spectral estimate \eqref{equ-spec-ball} so by Lemma 2.3, Lemma 2.4 in \cite{WWZZ19}, we have
		\begin{equation}\label{equ-obs-heat}
			\|e^{T\Delta}u_0\|_{L^2(\R^d)}\le e^{C(C_{spec}+1)^2(1+1/T)}\int_0^T\|e^{t\Delta}u_0\|_{L^2(\mathcal{O})}^2dt.
		\end{equation}
		Then, we consider the $d$-dim inhomogeneous backward heat equation:
		\begin{equation}\label{Equ-inhomo-heat}
			\partial_t u+(\Delta-1)u=F\in L^\infty((0,\infty);L^2(\R^d)),\ u|_{t=T}=u_T\in H^2(\R^d).
		\end{equation}
		Analogous to Corollary 3.12 in \cite{SSY23},
		by \eqref{equ-obs-heat} and reversing the time $t$ to $T-t$, there exists a constant $C>0$, depending only on $d,\,\mathcal{O}$, such that for any $T>0$ and any solution $u$ of \eqref{Equ-inhomo-heat}, we have 
		\begin{equation}\label{equ-obs-inhomo-heat}
			\|u_0\|_{L^2(\R^d)}^2\le Ce^{\frac{C}{T}}\int_{0}^{T}\big(\|(-\Delta+1)u(t)\|_{L^2(\mathcal{O})}^2+\|F(t)\|_{L^2(\R^d)}^2\big)dt.
		\end{equation}
		
		\emph{Step 2: Observability for Schr\"{o}dinger equation in high-frequency part.} 
		By the resolvent estimate \eqref{equ-resolvent} and the properties of $M(\lambda)$, arguing as in the proof of Corollary 4.4 in \cite{SSY23}, one obtains that there exist constants $\lambda_1,\,C>0$, depending only on $d,\mathcal{O}$, such that for any $T>0,\, \lambda>\lambda_1(1+T^{-1/2})$ and $f\in L^2(\R^d)$,
		\begin{equation}\label{equ-obs-high}
			\|(I-\Pi_\lambda)f\|_{L^2(\R^d)}^2\le \frac{C}{T}\int_{0}^{T}\|e^{it(-\Delta+1)}(I-\Pi_\lambda)f\|_{L^2(\mathcal{O})}^2dt.
		\end{equation}
		Here, the spectral projector $\Pi_\lambda=\mathbbm{1}_{\sqrt{-\Delta}\le\lambda}$ is defined by
		\[
		\Pi_\lambda f= \mathcal{F}^{-1}\chi_{\{|\xi|\le\lambda\}}\mathcal{F} f.
		\]
		
		\emph{Step 3: We complete the proof.} Similarly, we introduce the FBI transformation.
	 For $0<h<1,\,z=\tau+is\in\C$ and $L^2(\R^d)$-valued regular function $\Gamma(t;x)$, we define
		\begin{equation*}
			\mathcal{T}_h\Gamma(z;x)=\frac{2^{1/4}}{(2\pi h)^{3/4}}\int_\R e^{-\frac{(z+t)^2}{2h}}\Gamma(t;x)dt.
		\end{equation*}
		Then, similarly to Proposition 5.1 in \cite{SSY23}, using the FBI transform together with standard computations, we derive the following observability estimate based on \eqref{equ-obs-inhomo-heat}:
		\begin{equation}\label{equ-obs-schro2}
			\|f\|_{L^2(\R^d)}^2\le Ch(1+\frac{1}{T^2})\|(-\Delta+1)f\|_{L^2(\R^d)}^2+C e^{\frac{2T^2}{h}+\frac{C}{T}}\int_{0}^{T}\|(-\Delta+1)e^{it(-\Delta+1)}f\|_{L^2(\mathcal{O})}^2dt,
		\end{equation}
		for any $T>0,\,0<h<h_0(1+T^{-3})^{-1}$ and $f\in H^2(\R^d)$.
		
		Next, for any $u_0\in L^2(\R^d)$, we denote $U_0=(-\Delta+1)^{-1}u_0\in H^2(\R^d)$ and $u(t)=e^{it(-\Delta+1)}u_0$.
		By \eqref{equ-obs-schro2}, we have
		\begin{equation}\label{equ-U0}
			\|U_0\|_{L^2(\R^d)}^2\le Ch(1+\frac{1}{T^2})\|u_0\|_{L^2(\R^d)}^2+C e^{\frac{2T^2}{h}+\frac{C}{T}}\int_{0}^{T}\|u(t)\|_{L^2(\mathcal{O})}^2dt,
		\end{equation}
		for any $T>0,\,0<h<h_0(1+T^{-3})^{-1}$. Then, by \eqref{equ-obs-high} and the triangle inequality, we obtain
		\begin{equation*}
			\|u_0\|_{L^2(\R^d)}^2\le \frac{C}{T}\int_{0}^{T}\|u(t)\|_{L^2(\mathcal{O})}^2dt+(C+1)\|\Pi_\lambda u_0\|_{L^2(\R^d)}^2,
		\end{equation*}
		for all $T>0$ and $\lambda>\lambda_1(1+T^{-1/2})$.
		From the definitions of $U_0$ and $\Pi_\lambda$, we obtain
		\begin{equation*}
			\|\Pi_\lambda u_0\|_{L^2(\R^d)}^2=\int_{|\xi|\le\lambda}|\hat{u}|^2d\xi\le (\lambda^2+1)^2\int_{|\xi|\le\lambda}|(|\xi|^2+1)^{-1}\hat{u}(\xi)|^2d\xi\le(\lambda^2+1)^2\|U_0\|_{L^2(\R^d)}^2.
		\end{equation*}
		Thus, we derive that
		\begin{equation*}
			\|u_0\|_{L^2(\R^d)}^2\le \frac{C}{T}\int_{0}^{T}\|u(t)\|_{L^2(\mathcal{O})}^2dt+C\lambda^4\|U_0\|_{L^2(\R^d)}^2.
		\end{equation*} 
		This, together with \eqref{equ-U0}, implies
		\begin{equation*}
			\|u_0\|_{L^2(\R^d)}^2\le C\lambda^4 h(1+\frac{1}{T^2})\|u_0\|_{L^2(\R^d)}^2+\big(C\lambda^4e^{\frac{2T^2}{h}+\frac{C}{T}}+\frac{C}{T}\big)\int_{0}^{T}\|u(t)\|_{L^2(\mathcal{O})}^2dt.
		\end{equation*}
		By taking $\lambda=2\lambda_1(1+T^{-1/2})$ and $h=\varepsilon(1+T^{-4})^{-1}$ with $\varepsilon$ small enough, we complete the proof.
	\end{proof}

	\bibliographystyle{alpha}
	\small{\bibliography{GeometricConditionsSchrodingerV2}}
	
\end{document}